\documentclass[12pt, twoside, leqno]{article}

\usepackage{amsmath,amsthm}
\usepackage{amssymb}

\usepackage{enumitem}

\usepackage{graphicx}

\usepackage[T1]{fontenc}
\usepackage[utf8]{inputenc}
\usepackage{amsmath,amssymb,amsthm,latexsym,graphicx}
\usepackage[normalem]{ulem}
\usepackage{setspace} %used for doublespacing, etc.
\usepackage{hyperref}
\usepackage{cancel}
\RequirePackage{doi}
\usepackage{hyperref}
\usepackage{doi}

\usepackage{titlesec}
 
\usepackage{lipsum}
\allowdisplaybreaks
 
\usepackage{url}
\newcommand{\RR}{{\mathbb R}}

\newcommand{\ZZ}{{\mathbb Z}}

\newcommand{\bc}{\begin{center}}
\newcommand{\ec}{\end{center}}
\newcommand{\be}{\begin{enumerate}}
\newcommand{\ee}{\end{enumerate}}
\newcommand{\bi}{\begin{itemize}}
\newcommand{\ei}{\end{itemize}}

\newcommand{\beq}{\begin{equation}}
   
\newcommand{\eeq}{\end{equation}}
\newtheorem{exmp}{Example}
\newcommand{\bex}{\begin{exmp}}
	\newcommand{\eex}{\end{exmp}}

\newcommand{\pdot}{{p(\cdot)}}

\newcommand{\Sol}[1] {\textbf{Solution:}}

\newcommand {\abs}[1] {\lvert#1\rvert}

\newcommand{\normb}[1]{\ensuremath{\left\lVert#1\right\rVert }}

\newcommand\frightarrow{\scalebox{1}[.3]{$\rule[.45ex]{2ex}{1.5pt}%
		\kern-.2ex{\blacktriangleright}$}}

\newcommand{\vertiii}[1]{{\left\vert\kern-0.25ex\left\vert\kern-0.25ex\left\vert #1 
		\right\vert\kern-0.25ex\right\vert\kern-0.25ex\right\vert}}

\newcommand{\ellw}{\ell_w}

\newcommand{\seq}[1]{\ensuremath{ \left \{ #1(n): n \in \ZZ \right \}}}

\allowdisplaybreaks

\newtheorem{theorem}{Theorem}[section]

\newtheorem{lemma}[theorem]{Lemma}

\theoremstyle{definition}
\newtheorem{definition}[theorem]{Definition}

\numberwithin{equation}{section}

\begin{document}
\baselineskip=17pt
% \title{ Boundedness of Maximal Singular operator of Calderon-Zygmund type for Variable Sequence Spaces $\lpdot (\ZZ)$ }

\title{{\bf Maximal Ergodic Theorem On Weighted $L^p_w(X)$ spaces}}
\author{Sri Sakti Swarup Anupindi\\
Department of Mathematics\\ 
 Birla Institute of Technology and Science-PILANI, Jawahar Nagar\\
Hyderabad-500 078, Telangana, India.\\
E-mail: p20180442@hyderabad.bits-pilani.ac.in
\and 
A.Michael Alphonse\\
Department of Mathematics\\ 
 Birla Institute of Technology and Science-PILANI, Jawahar Nagar\\
Hyderabad-500 078, Telangana, India. \\
E-mail: alphonse@hyderabad.bits-pilani.ac.in}

\date{\today}
\maketitle

\renewcommand{\thefootnote}{}
\let\thefootnote\relax

\let\thefootnote\relax\footnotetext{Calderon-Zygmund decomposition, Maximal Ergodic Operator, Transference Method, Ergodic Rectangles, Ergodic Weights. }

\begin{abstract}
In this paper, we study the maximal ergodic operator on $L^p_w(X, \mathcal{B}, \mu)$ spaces, $1 \leq p < \infty$, where $(X, \mathcal{B}, \mu)$ is a probability space equipped with an invertible measure preserving transformation $U$ and $w$ is an ergodic $A_p$ weight using transference method.
% Using these weighted results and Rubio de Francia extrapolation method, we obtain the boundedness results on variable $L^{\pdot}(X, \mathcal{B}, \mu)$ spaces, $ 1 \leq p < \infty$.
\end{abstract}

\section{Introduction}
In this paper, using transference method, we prove strong type, weak type inequalities for maximal ergodic operators on $L^p_w(X, \mathcal{B}, \mu)$ spaces, $1 \leq p < \infty$, where $(X, \mathcal{B}, \mu)$ is a probability space equipped with an invertible measure preserving transformation $U$ and $w$ is an ergodic $A_p$ weight.
% Ergodic singular operators on $\ell^p(\ZZ)$ without weights were studied in $\text{\cite{Mich1}}$
%%We also prove boundedness of maximal ergodic operator on variable exponent spaces using Rubio de Francia extrapolation method and the corresponding weighted results.

%The results of Hardy-Littlewood maximal operator on variable $L^p$ spaces can be seen in $\text{\cite{fior_new_proof_maximal}}, \text{\cite{David2}}, \text{\cite{David3}},%\text{\cite{David4}},\text{\cite{David5}},\text{\cite{fior_var_book}},\text{\cite{lars1}}$. 
For standard results on boundedness of various maximal operators in harmonic analysis
% including classical weighted spaces i.e $L_w^{p}(\RR)$, 
we refer to $\text{\cite{Duan1}}$.
Boundedness of  maximal ergodic operator for $L^p$ spaces with weights can be found in $\text{\cite{Att1}}$. In $\text{\cite{Att1}}$ the characterization of those positive functions $w$ such that the maximal ergodic operator associated with an invertible measure preserving transformation on a probability space is a bounded operator on $L^p(wd\mu)$ is given. In their proof the ergodic analogue of Calderon-Zygmund decomposition and the concept of ergodic rectangles are used. In this paper we prove the same result using Calderon-Coifman-Weiss transference principle $\text{\cite{Cald1}}, \text{\cite{Coif2}}, \text{\cite{Mich2}}$.

%The similar transference techniques can be found in  $\text{\cite{Cald1}},\text{\cite{Coif1}},\text{\cite{Coif2}},\text{\cite{Petersen1}},  \text{\cite{Cotlar1}},  \text{\cite{Mich1}},  \text{\cite{Mich2}}$.
%In $ \text{\cite{Mich2}}$, it shown that the maximal ergodic singular operator is bounded on $\ell^p(X)$ for $ 1 < p < \infty$ and is of weak type (1,1) using transference method.
%
%In this paper, we show that the maximal ergodic operator for $1 < p < \infty$ is bounded on $L^{\pdot}(X)$, where $(X, \mathcal{B}, \mu)$ is a probability space equipped with measure preserving transformation $U$ with appropriate use of ergodic $A_p$ weights.
%  
 
 \section{Notation}\label{Sec 1}
Throughout this paper, $\ZZ$ denotes set of all integers and $\ZZ_{+}$ denotes set of all positive integers. For a given interval $I$ in $\ZZ$ (We always mean finite interval of integers) , $\abs{I}$ always denotes the cardinality of $I$. 
For each positive integer N, consider collection of disjoint intervals of cardinality $2^N$, 
\[ \left \{ I_{N,j} \right \}_{j \in \ZZ} = \left \{ [ (j-1)2^N+1, \dots, j2^N ] \right \}_{ j \in \ZZ}  \]
The set of intervals which are of the form $I_{N, j}$ where $ N \in \ZZ_{+}$ and $j \in \ZZ$ are 
%$\left\{ I_{N,j}: N \in \ZZ_{+} \right \}$ 
called dyadic intervals. For fixed $N$, $I_{N,j}$ are disjoint.
% \in \ZZ_{+}$, define $\mathcal{I}_N$ as family of dyadic intervals $\left \{ I_{N,j} \right \}, j \in \ZZ$.
Given a dyadic interval $I = \left \{ [ (j-1)2^N+1, \dots, j2^N ] \right \}_{ j \in \ZZ} $ and a positive integer $m$, we define
 \begin{align*}
2mLI & = [ (j-(2m-1))2^N+1, \dots, j2^N ] \\
2mRI & = [ (j-1)2^N+1, \dots, (j+(2m-1))2^N ] \\
(2m+1)I& = 2mLI \cup 2mRI.
 \end{align*}
%Note $\abs{3I}=3.2^N$. 
Note that $2mLI, 2mRI$ are dyadic intervals each of length $(2m)2^N$.
%, $2^{N+1}$, while $4LI, 4RI$ are dyadic intervals each of length $3(2^{N})$.
However $(2m+1)I, m \in \ZZ_{+}$  are not  dyadic intervals. 

\section{Definitions}\label{sec 2}
\subsection*{Maximal Operators}
Let $\seq{a}$ be a sequence. We define the following three types of Hardy-Littlewood maximal operators as follows:
\begin{definition}
 If $I_r$ is the interval $\left \{-r,-r+1, \dots, 0, 1, 2, \dots,r-1,r \right\} $, define centered Hardy-Littlewood maximal operator
\[ M^\prime a(m) = \sup_{r >0} \frac{1}{(2r+1)}  \sum_{n \in I_r} \abs{a(m-n)}    \]
For any positive integer $J$, define truncated centered Hardy-Littlewood maximal operator as
\[ (M^\prime_J a) (m) = \sup_{J> r >0} \frac{1}{(2r+1)}  \sum_{n \in I_r} \abs{a(m-n)}    \]
\end{definition}
\begin{definition}
We define Hardy-Littlewood maximal operator as follows
\[ M a(m) = \sup_{m \in I} \frac{1}{\abs{I}} \sum_{ n \in I} \abs{a(n)}   \]
 where the supremum is taken over all intervals containing $m$.
 For any positive integer $J$, define truncated Hardy-Littlewood maximal operator as
 \[ M_J a(m) = \sup_{J \geq m \in I} \frac{1}{\abs{I}} \sum_{ n \in I} \abs{a(n)}   \]
\end{definition}

\begin{definition}
We define dyadic Hardy-Littlewood maximal operator as follows:
\[ M_d a(m) = \sup_{ m \in I} \frac{1}{\abs{I}} \sum_{ k \in I} \abs{a(k)} \]
where supremum is taken over all dyadic intervals containing $m$.
\end{definition}

\begin{definition}
 Given a sequence $\seq{a}$ and an interval $I$, let $a_I$ denote average of ${\seq{a}}$ on $I$.
Let, $ a_I = \frac{1}{\abs{I}} \sum_{ m \in I} a(m) $. Define the sharp maximal operator $M^{\#}$ as follows  
\[ M^{\#} a(m) = \sup_{ m \in I} \frac{1}{\abs{I}} \sum_{ n \in I} \abs{ a(n) - a_I} \]
where the supremum is taken over all intervals $I$ containing $m$. We say that sequence $\seq{a}$ has bounded mean oscillation if the
sequence $M^{\#}a$ is bounded. The space of sequences with this property is denoted by BMO($\ZZ$).
%\[ BMO  = \left \{ a \in \ell^1_{loc}: M^{\#} a \in \ell^\infty  \right \} \] 
We define a norm in BMO($\ZZ$) by $ \normb{a}_\star = \normb{M^{\#}a}_\infty$. The space BMO($\ZZ$) is studied in $\text{\cite{Mich2}}$.
\end{definition}

\subsection*{Weights}
\begin{definition}
For a fixed $p$, $1 < p < \infty$, we say that a non-negative sequence $\seq{w}$ belongs to class $A_p$ if there is a constant $C$ such that, for all intervals $I$ in $\ZZ$, we have
\[ \biggl( \frac{1}{\abs{I}} \sum_{k \in I} w(k) \biggr) \biggl( \frac{1}{\abs{I}} \sum_{k \in I} w(k)^{-\frac{1}{p-1}} \biggr)^{p-1} \leq C \]
This constant $C$ is called $A_p$ constant.
We say that $ \left \{ w(m) : m \in \ZZ \right \}$  belongs to class $A_1$  if there a constant C such that, for all intervals $I$ in $\ZZ$,
\[ \frac{1}{\abs{I}} \sum_{k \in I} w(k) \leq C w(m) \] for all $m \in I$. 
This constant $C$ is called $A_1$ constant.\\
Let $ 1 \leq p < \infty$ and $\seq{w} \in A_p(\ZZ)$. We say that a sequence $\seq{a}$ is in $\ell^p_w(\ZZ)$ if
\[ \sum_{n \in \ZZ} \abs{a(n)}^p w(n) < \infty\]
We define norm in $\ell^p_w(\ZZ)$ by
% For a given sequence $\seq{a} \in \ellw^p(\ZZ)$ and a non-negative weight sequence $\seq{w}$, we define weighted classical norm as
 \[ \normb{a}_{\ellw^p(\ZZ)} = \biggl( \sum_{k \in \ZZ} {\abs{ a(k)}}^{p} w(k) \biggr)^{\frac{1}{p}} \] 
For a given sequence $\seq{a} \in \ellw^p(\ZZ)$, the weighted weak(p,p) inequality for a non-negative weight sequence $\seq{w}$ is as follows:
%The weighted, weak (p,p) inequality for $M$ with respect to $w$ is 
\begin{align*} 
 w(\left \{ m \in \ZZ: Ma(m) > \lambda \right \}) \leq \frac{C}{\lambda^p} \sum_{m \in \ZZ} \abs{a(m)}^p w(m)  \tag{A4} \label{A4}
\end{align*}
For a subset $A$ of $\ZZ$, $w(A)$ denotes $\sum_{k \in A} w(k) $. 
\end{definition}
The following definition is from {\text{\cite{Att1}}\begin{definition}
Let $(X,\textbf{B}, \mu)$ be a probability space and $U$ an invertible measure preserving transformation on X.
Suppose $1 < p < \infty$ and $w:X \to \RR$ be a non-negative integrable function. The function $w$ is said to satisfy ergodic $A_p$ condition if \\
\[ esssup_{x \in X}    \sup_{N \geq 1} \biggl( \frac{1}{2N+1} \sum_{k=-N}^N w(U^k x) \biggr) \biggl( \frac{1}{2N+1} \sum_{k=-N}^N w(U^kx)^{\frac{-1}{p-1}} \biggr)^{p-1}  \leq C  \]
The function $w$ is said to satisfy ergodic $A_1$ condition if
\[ esssup_{x \in X} \sup_{N \geq 1} \frac{1}{2N+1} \sum_{k=-N}^N w(U^k x) \leq C w(U^m x) \]
for $m = -N, -N+1, \dots, N$
\end{definition}

\begin{definition}
Let $ 1 \leq p < \infty$. We say that a measurable function $f\in L^p_w(X)$  if
\[ \int_{x \in X} \abs{f(x)}^p w(x) d\mu(x) < \infty\]
We define norm in $L^p_w(X)$ by
% For a given sequence $\seq{a} \in \ellw^p(\ZZ)$ and a non-negative weight sequence $\seq{w}$, we define weighted classical norm as
 \[ \normb{f}_{L^p_w(X)} = \biggl( \int_{x \in X} {\abs{ f(x)}}^{p} w(x) d\mu(x) \biggr)^{\frac{1}{p}} \] 
\end{definition}

\section{ Relations between Maximal operators }\label{sec 3}
In the following lemmas, we give relations between maximal operators. For the proofs of the following lemmas, refer $\cite{anup1}$. These relations will be used when we prove the weighted inequalities for maximal ergodic operators.
\begin{lemma}{\label{Sak_lem_2.11}}
Given a sequence $\left \{ a(m): m \in \ZZ \right \}$, the following relation holds:
\[   M^{\prime}a(m)  \leq M a(m) \leq 3 M^\prime a(m) \]
\end{lemma}
 
%\end{proof}

\begin{lemma}{\label{Sak_lem_2.12}}
If $\textbf{a}= \left \{ a(k): k \in \ZZ \right \}$ is a non-negative sequence with $\textbf{a} \in \ell_1$, then 
\[ \abs{ \left \{  m \in  \ZZ: M^\prime a(m) > 4  \lambda \right \} } \leq 3  \abs{ \left \{ m \in \ZZ: M_d a(m) > \lambda \right \} } \]
\end{lemma}
In the following lemma, we see that in the norm of BMO($\ZZ$) space, we can replace the average $a_I$ of $\left \{ a(n) \right \}$ by a constant $b$. The proof is similar to the proof in continuous version $\cite{Duan1}$. The second inequality follows from $\abs { \abs{a} - \abs{b} } \leq \abs{a} - \abs{b} $.
\begin{lemma}{\label{Sak_Prop_6.5}} 
Consider  a non-negative sequence $\textbf{a}= \left \{ a(k): k \in \ZZ \right \}$. Then the following are valid.
\begin{align*}
1.& \quad \frac{1}{2} \normb{a}_{\star} \leq \sup_{m \in I} \inf_{b \in \ZZ} \frac{1}{\abs{I}} \abs{ a(m)- b} \leq \normb{a}_\star \\
2.& \quad M^{\#}(\abs{a})(i) \leq M^{\#} a(i) , i \in \ZZ
\end{align*}

\end{lemma}

\section{Weighted Classical Results for Maximal Operators}

In this section, for a given sequence $\seq{a}$ in $\ell^p_w(\ZZ)$, we prove weighted weak(p,p) inequality with respect to the weight $\seq{w} \in A_p$ which is as follows:
%The weighted, weak (p,p) inequality for $M$ with respect to $w$ is 
\begin{align*} 
 w(\left \{ m \in \ZZ: Ma(m) > \lambda \right \}) \leq \frac{C}{\lambda^p} \sum_{m \in \ZZ} \abs{a(m)}^p w(m)  \tag{A4} \label{A4}
\end{align*}

%The discussion in this section assumes that $\ref{Sak_7.1}$ is satisfied, which is true for any $ w \in A_p, 1 \leq p < \infty$. This is indeed theorems $\ref{Sak_Thm_2.16}$  and $\ref{Sak_Thm_7.1}$ which are quoted below for cases $ p=1$ and $1 < p< \infty$.
Inequality $\ref{A4}$ will be proved via several theorems, Theorem [$\ref{Sak_Thm_2.16}$] to Theorem [$\ref{Sak_Thm_7.1}$]. 

The proof of Theorem$\ref{Sak_Thm_2.16}$ uses Calderon-Zygmund decomposition.
For the proofs of the corresponding results for the continuous version, we refer $\cite{Duan1}$. The proofs of Theorem [$\ref{Sak_Thm_2.16}$] and Theorem [$\ref{Sak_Thm_7.1}$] are same as the proof for the continuous versions of the corresponding results apart from the fact that the constants obtained here are slightly different from the constants obtained for the continuous version due to the nature of dyadic intervals in $\ZZ$ [See $\cite{anup2}$]. So, we give here the proof of Theorem[$\ref{Sak_Thm_2.16}$] and Theorem $\ref{Sak_Thm_7.1}$ for the sake of completeness.

%Since, the Calderon-Zygmund decomposition for integers uses dyadic intervals in $\ZZ$, the proof of Theorem$\ref{Sak_Thm_2.16}$ slightly differs %from continuous version. So, we give the proof.
\begin{theorem}{\label{Sak_Thm_2.16}}
If $\seq{w} \in A_p$ , $ 1 \leq  p < \infty$ and $\seq{a}$ is a sequence in $\ell^p(\ZZ)$, then for $p>1$, there exists a constant $C_p$ such that
\[ \sum_{m \in \ZZ} Ma(m)^pw(m)  \leq C_p \sum_{m \in \ZZ} \abs{a(m)}^p Mw(m)  \tag{\ref{Sak_Thm_2.16}[A]} \label{5.1A} \]
Furthermore, for $p=1$, there exists a constant $C_1$ such that
\[ \sum_{\left \{ m \in \ZZ:Ma(m) > \lambda \right \}} w(m)   \leq \frac{C_1}{\lambda} \sum_{m \in \ZZ} \abs{a(m)} Mw(m) \tag{\ref{Sak_Thm_2.16}[B]} \label{5.1B} \]
\end{theorem}

\begin{proof}
We will show that $\normb{Ma}_{\ell^\infty_w(\ZZ)} \leq \normb{a}_{\ell^\infty_{Mw}(\ZZ)}$ and that weak(1,1) inequality holds; the strong(p,p) inequality then follows from the
Marcinkiewicz interpolation theorem. \\
%%%% New Proof Below
Take $C > \normb{a}_{\ell^\infty_{Mw}(\ZZ)}$. Then 
\[ \sum_{\left \{ n \in \ZZ: \abs{a(n)} > C \right \} } Mw(n) =0 \] which shows that
\[ \left \{ n \in \ZZ: \abs{a(n)} > C  \right\} = \emptyset \] Hence $\abs{a(n)} \leq C, \forall n \in \ZZ$ which implies that $\abs{Ma(n)} \leq C, \forall n \in \ZZ $  Therefore
\[ \left \{ n \in \ZZ: \abs{Ma(n)} > C \right \} = \emptyset \]
So, \[ \sum_{\left \{ n \in \ZZ: \abs{Ma(n) } > C \right \} } w(n) =  0 \]
which gives $ w( \left \{ n \in \ZZ: \abs{Ma(n)} > C \right \} )=0 $. \\
Therefore $ \normb{Ma}_{\ell^\infty_w(\ZZ)} \leq C$.
Taking $\inf \left \{ C: \normb{a}_{\ell^\infty_{Mw}(\ZZ)} < C \right \}$
we get 
\[ \normb{Ma}_{\ell^\infty_w(\ZZ)} \leq \normb{a}_{\ell^\infty_{Mw}(\ZZ)} \]

To prove the weak(1,1) inequality we may assume that $\seq{a} \in \ell^1(\ZZ)$. Form Calderon-Zygmund decomposition of sequence $\seq{a}$ at height $\frac{\lambda}{4} > 0 $. Then we get a sequence $\left \{ I_j \right \} $  of dyadic intervals in $\ZZ$ such that 
%If $ \left\{ I_j \right \}$ is the Calderon-zygmund decomposition of sequence $\textbf{a}= \left \{ a(n) \right \}$ at height %$\frac{\lambda}{4} >0$, then
\[ \frac{1}{\abs{I_j}} \sum_{k \in I_j} \abs{a(k)} > \frac{\lambda}{4} \]
Further as we showed in the proof of Lemma [$\ref{Sak_lem_2.12}]( See$ \cite{anup1}), 
%with $M=M^{\prime}$
\[ \left \{ m \in \ZZ: M^\prime a(m) >  \lambda \right \} \subset \cup_j 3I_j \]
%Given that $ w \in A_1$, 
It follows that
\begin{align*}
\sum_{ \left \{ m: M^\prime a(m) >  \lambda \right \} } w(m)   & \leq \sum_j \sum_{m \in 3I_j} w(m)   \\
&= \sum_j 3 \abs{I_j} \frac{1}{\abs{3I_j}} \sum_{m \in 3I_j} w(m)   \\
&\leq \frac{12}{\lambda} \sum_j \biggl(  \biggl( \sum_{n \in I_j} \abs{a(n)} \biggr) \biggl( \frac{1}{\abs{3I_j}} \sum_{m \in 3I_j} w(m)   \biggr) \biggr)   \\
&\leq \frac{12 C}{\lambda} \sum_{n \in \ZZ} \abs{a(n)} Mw(n)
% \\   
%&\leq \frac{3 C}{\lambda} \sum_{n \in \ZZ} a(n) C w(n)
\end{align*}
Since by Lemma[$\ref{Sak_lem_2.11}$], $\left \{ m: Ma(m) > \lambda \right \} \subseteq \left \{ m: M^\prime a(m) > \frac{\lambda}{3} \right \}$, it follows that
\[ \sum_{\left \{m: Ma(m)>\lambda \right \}} w(m) \leq \sum_{ \left \{ m: M^\prime a(m) > \frac{\lambda}{3} \right \}} w(m) \leq \frac{ 36C}{\lambda} \sum_{n \in \ZZ} \abs{a(n)} Mw(n) \]
%Since $M^{\prime \prime} w(n) \leq C_d Mw(n) $, we get the desired inequality.
%If $w$ is such that $ Mw(n) \leq w(n) $ a.e, then these inequalities hold with the same weight $w$ on both sides . Sequences, which satisfy this condition are called $A_1$ weights. The right hand side of the weak (1,1) inequality in this theorem contains $Mw$ and since $ w \in A_1$, we can replace $Mw$ by $Cw$. 
\end{proof}

The proof of the following theorem is similar to the proof of corresponding result in continuous version $\cite{Duan1}$. We state here without proof.
\begin{theorem}{\label{Sak_own_7.1}}
Let ${\seq{a}}$ be a non-negative sequence and $\seq{w} \in A_p, 1\leq p < \infty$ be a non-negative weight sequence. Let $I$ be an interval  such that $a(m) > 0$ for some $m \in I$. Then, 
\begin{enumerate}
\item \[ w(I) \biggl( \frac{a(I)}{\abs{I}} \biggr)^p \leq C \sum_{ m \in I} \abs{a(m)}^p w(m)  \tag{\ref{Sak_own_7.1}[A] }    \]
\item Given a finite set $S \subset I$ ,\[ w(I) \biggl( \frac{\abs{S}}{\abs{I}} \biggr)^p \leq Cw(S) \tag{\ref{Sak_own_7.1}[B] }   \]
\end{enumerate}
\end{theorem}
$\ref{Sak_own_7.1}[A]$ follows from Holder's inequality and the $A_p$ condition.
$\ref{Sak_own_7.1}[B]$ follows by taking $ a =\chi_S$ in $\ref{Sak_own_7.1}[A]$.
% \begin{proof}

\begin{theorem}{\label{Sak_Thm_7.1}}
Assume $\seq{w} \in A_p$. Given a non-negative sequence $\left \{ a(n): n \in \ZZ \right \} \in \ellw^p(\ZZ)$, for $1 \leq p< \infty$, the weighted weak(p,p) inequality holds:
\[ w(\left \{ m \in \ZZ: Ma(m) > \lambda \right \}) \leq \frac{C}{\lambda^p} \sum_{m \in \ZZ} \abs{a(m)}^p w(m)  \]
%holds if and only if $ w \in A_p$.
\end{theorem}

% \[ Rh_1(x) = \sum_{k=0}^\infty \frac{M^k h_1(x) }{2^k \normb{M}^k_{L^{\pplus}}} \]

% \[ Rh_2(x) = \sum_{k=0}^\infty \frac{M^k h_2(x) }{2^k \normb{M}^k_{L^{\pplus}}} \]

\begin{proof}
 
%\end{align*}

Let $ \seq{a} \in \ellw^p(\ZZ)$.
%We may also assume that $ a \in \ell^1$ since otherwise we can replace $a$ by $ a \chi_{B(0,k)}$ and the following argument will yield constants independent of $k$. 
Form the Calderon-Zygmund decomposition of $\seq{a}$ at height $\frac{\lambda}{12}$ to get a collection of disjoint intervals $\left \{ I_j \right \}$ such that $ a(I_j) > \frac{\lambda}{12} \abs{I_j}$. By the proof of Lemma$[{\ref{Sak_lem_2.12}}]$ in $\cite{anup1}$ and Lemma[$\ref{Sak_lem_2.11}]$ , we have
%Then by the same argument as in the proof of $\ref{Sak_em_2.12}$, 
\[ \left \{ m \in \ZZ: Ma(m) > \lambda \right \} \subseteq \left \{ m \in \ZZ: M^\prime a(m) > \frac{\lambda}{3} \right \} \subseteq \cup_j 3I_j \]
%Here we dilate the cubes by a factor of 3 instead of 2 because M is the non-centered Hardy-Littlewood maximal operator.
% $M$ by definition for weighted inequalities is non-centered maximal opeartor in Duandikotoxea on page 133.
Therefore, using Theorem[$\ref{Sak_own_7.1}$], we have 
%$\ref{A5}$
\begin{align*}
w(\left \{ m \in \ZZ: Ma(m) > \lambda \right \} ) & \leq \sum_j w(3I_j) \leq C 3^p \sum_j w(I_j) \\
&  \leq C 3^p  \sum_j \biggl( \frac{\abs{I_j}} {a(I_j)} \biggr)^p \sum_{m \in I_j} \abs{a(m)}^p w(m) \\
& \leq C 3^p  \biggl( \frac{12}{\lambda} \biggr)^p \sum_{m \in \ZZ} \abs{a(m)}^p w(m) 
\end{align*}
%where the second inequaliity follows from $\ref{Sak_7.3}$ and the third from $\ref{Sak_7.2}$ and the fourth from the properties of $I_j$.
\end{proof}

\begin{theorem}{\label{Sak_Thm_7.3}} 
If $ w \in A_p, 1 < p < \infty$, then $M$ is bounded on $\ellw^p(\ZZ)$.
\end{theorem}
The proof follows from Theorem[$\ref{Sak_Thm_7.1}$] and Marcinkiewicz interpolation theorem.
%\begin{proof}  Sketch: . For a given $ w \in A_p$, by Reverse Holder inequality there exists $ 1 < q< p$ such that $w \in A_q$. so, let  $w \in A_q$, where $1 < q < p$. Using Marcinkiewicz interpolation theorem, we interpolate between weak $\ell_\infty(w)$ and  weak(q,q)inequalities to arrive at strong(p,p) inequality for all $p > q$. Since strong(p,p) inequality implies weak(p,p) inequality, the results follows.
%\end{proof}

 \section{ Maximal Ergodic Operator}
Let $(X,\textbf{B},\mu)$ be a probability space and $U$ an invertible measure preserving transformation on $X$.
We define maximal ergodic operator as 
\[ \tilde{M} f(x) = \sup_{J \geq 1} \frac{1}{2J+1} \sum_{k =-J}^J  \abs{f(U^{-k} x)  } \]

For any positive integer $J$, we also define truncated maximal ergodic operator as 
\[ \tilde{M}_J f(x) = \sup_{ 1 \leq n \leq J} \frac{1}{2n+1} \sum_{k =-n}^n  \abs{f(U^{-k} x)   } \]
%
% 
%Given
%\[ \tilde{M} f(x) = \sup_{N \geq 1} \frac{1}{2N+1}\sum_{k =-N}^N \abs{  f(U^{-k} x)  } \]
%Define
%\begin{align*}
%f(U^{-k} x)   = \begin{cases}
%	 \,  a(k)  \quad if  -J \leq k \leq J \quad \text{for a.e} \quad x \in X \quad \text{and} \quad k \in \ZZ_{+} \\
%	  \, 0 \quad {otherwise} 
%	\end{cases}
%\end{align*}
%Then,
%\begin{align}\label{Sak_eq1}
%\tilde{M^\star} f(U^{-j}x) &= \sup_{J \geq 1} \frac{1}{2J+1}\sum_{k =-J}^J \abs{  f(U^{-j-k} x)  }  \nonumber \\
%&= \sup_{J \geq 1} \frac{1}{2J+1}  \sum_{k =-J}^J \abs{a(j+k)  } = M^\prime a(j) 
%\end{align}
%
%Note similarly like above for the operator defined below \[ \tilde{M^\star_J} f(x) = \sup_{J \geq n \geq 1} \frac{1}{2n+1} \sum_{k =-n}^n  \abs{f(U^{-k} x)  } \] we have
%\begin{align}\label{Sak_eq2}
%\tilde{M_J} f(U^{-j}x) &= \sup_{J \geq n \geq 1} \frac{1}{2n+1}\sum_{k =-n}^n\abs{  f(U^{-j-k} x)  } \nonumber \\
%&= \sup_{J \geq n \geq 1} \frac{1}{2n+1}  \sum_{k =-n}^n \abs{a(j+k)  } ={ M^\prime_J} a(j) 
%\end{align}
%\end{proof}

In the following theorem using transference, we prove that the maximal ergodic operator is bounded on $L^p_w(X, \mathcal{B}, \mu) , 1 <  p < \infty$ where $w$ is ergodic $A_p$ weight and the  maximal ergodic operator satisfies weak type (1,1) inequality on $L^1_w(X, \mathbf{B}, \mu)$ space.
 
\begin{theorem}{\label{Sak_thm_Ergodic_oneway}}
Let $(X,\textbf{B},\mu)$ be  a probability space and $U$ an invertible measure preserving transformation on $X$.
%\[ \tilde{M}^\star_N f(x) = \sup_{ 1 \leq n \leq N} \abs{ \sum_{k =-n}^n f(U^{-k} x)   } \]
satisfies
\begin{enumerate}
    \item   If $w$ is an ergodic $A_p$  weight, $ 1 <  p < \infty$ and $f \in L^p_w(X,\mathcal{B}, \mu) $, then the maximal ergodic operator \[ \normb{\tilde{M} f(x)}_{L^p_w(X)} \leq C_p \normb{f}_{L^p_w(X)} \quad \text{if} \quad 1 < p < \infty \]
    \item If w is an ergodic $A_1$ weight and $f \in L^1_w(X,\mathcal{B}, \mu) $ , then \[\int_{\left \{ x \in X: \abs{\tilde{M} f(x)} > \lambda \right \} } w(x) d\mu(x) \leq \frac{C}{\lambda} \int_X \abs{f(x)} w(x) d\mu(x) \]
%    \item \[ \mu( \left \{ x \in X: \tilde{M}^\star f(x) > \lambda \right \} \leq \frac{C}%{\lambda} \normb{f}_{1,w} \forall f \in L^1(X) \text{and} \quad \lambda > 0 \]
\end{enumerate}
\end{theorem}
\begin{proof}
Take $p ,  1\leq p < \infty$. Fix $J > 0$ and take a function $f \in L^p_w(X)$. 
\[ \tilde{M}_J f(x) = \sup_{ 1 \leq n \leq J} \frac{1}{2n+1}\sum_{k =-n}^n  \abs{ f(U^{-k} x)   } \]
It is enough to prove that $\tilde{M_J}$ satisfies (1) and (2) with constants not depending on $J$. Let $\lambda >0$ and put
\[ E_\lambda = \left \{ x \in X: \abs{\tilde{M_J} f(x)} > \lambda \right \} \]
%Below working not needed
%Since U is measure preserving, we have 
%%$\mu(E_N) = \mu(U^{-m} E_N) \quad \forall m \in \ZZ$. 
%\begin{align*}
%\mu(E_J) & = \mu(U^{-m} E_\lambda)  \quad \forall \quad m \in \ZZ\\
%& = \frac{1}{2M+1} \sum_{m=-M}^M \mu(U^{-m} E_\lambda) \quad \forall M \\
%& = \frac{1}{2M+1} \sum_{m=-M}^M \mu(\left \{ x : \tilde{M}^\star_J f(U^m x) > \lambda \right \}) \\
%& = \frac{1}{2M+1} \sum_{m=-M}^M \mu(\left \{ x : \sup_{1 \leq n \leq J} \frac{1}{2n+1} \sum_{k=-n}^n \abs{ f(U^{-k+m} x)} > \lambda \right \})
%\end{align*}
For x lying outside a $\mu$ null set and a positive integer $L$,  define sequences
\begin{align*}
a_x(k)  = \begin{cases}
	 \,  f(U^{-k} x)  \quad if  \quad \abs{k} \leq L + J  \\
	  \, 0 \quad {otherwise} 
	\end{cases}
\end{align*}

\begin{align*}
w_x(k) = \begin{cases}
	 \,  w(U^{-k} x)  \quad if  \quad \abs{k} \leq L + J  \\
	  \, 0 \quad {otherwise} 
	\end{cases}
\end{align*}
%By Equation[$\ref{Sak_eq1}$] and Lemma[$\ref{Sak_lem_2.11}$]
Using  Lemma[$\ref{Sak_lem_2.11}$], observe that for an integer $m$ with $\abs{m}\leq L$
\[ \tilde{M}_J f(U^{-m} x) = M^\prime_J a_x(m) \leq M a_x(m) \]
%correct version
% Fix $N > 0$ and let 
% \[ \tilde{M}^\star_N f(x) = \sup_{ 1 \leq n \leq N} \abs{ \sum_{k =-n}^n f(U^{-k} x)   } \]

% \[ E_n = \left \{ x \in X: \tilde{M}^\star_N f(x) > \lambda \right \} \]
Therefore,
\begin{align*}
  & w(\left \{ x \in X: \abs{\tilde{M}_J f(x) } > \lambda \right \})    = \int_{E_\lambda}  w(x) d\mu(x) = \frac{1}{\lambda^p} \int_{E_\lambda} \lambda^p w(x) d\mu(x) \\
  & \leq \frac{1}{\lambda^p} \int_{E_\lambda} \abs{ \tilde{M}_J f(x)}^p  w(x) d\mu(x) \\
    & \leq \frac{1}{\lambda^p} \int_X \abs{ \tilde{M}_J f(x)}^p  w(x) d\mu(x) \\
%  & \leq \frac{1}{\lambda^p} \int_{E_J} \abs{ \tilde{M}^\star_J f(x)}^p  w(x) d\mu(x) \\
  & = \frac{1}{\lambda^p} \frac{1}{2L+1} \sum_{m =-L}^L \int_X \abs{ \tilde{M}_J f(U^{-m} x)}^p  w(U^{-m} x) d\mu(x) \\
     & \leq  \frac{1}{\lambda^p} \frac{1}{2L+1} \sum_{m =-L}^L \int_X \abs{  {M}  a_x(m) }^p  w_x(m) d\mu(x) \\
     & =  \frac{1}{\lambda^p} \frac{1}{2L+1}  \int_X \sum_{m =-L}^L \abs{  {M}  a_x(m) }^p  w_x(m) d\mu(x) \\
     & \leq  \frac{1}{\lambda^p} \frac{1}{2L+1}  \int_X \sum_{m =-\infty}^\infty \abs{  {M}  a_x(m) }^p  w_x(m) d\mu(x) \\
     & \leq  \frac{1}{\lambda^p} \frac{1}{2L+1}  \int_X \sum_{m =-\infty}^\infty \abs{    a_x(m) }^p  w_x(m) d\mu(x) \\
%   & = \frac{1}{\lambda^p} \frac{1}{2M+1} \sum_{m =-M}^M \int_{E_\lambda} \abs{ \tilde{M}^\star_N a_x^M(m) }^p  w_x^M(m) d\mu(x) \\
%  & \leq \frac{1}{\lambda^p} \frac{1}{2M+1}  \int_{E_\lambda}    \sum_{m =-M}^{M}\abs{  a_x^M(m) }^p  w_x^M(m) d\mu(x) \\ 
& = \frac{1}{\lambda^p} \frac{1}{2L+1}  \int_X    \sum_{m =-(L+J)}^{(L+J)}   \abs{    a_x(m) }^p  w_x(m) d\mu(x) \\
    & = \frac{1}{\lambda^p} \frac{1}{2L+1}  \int_X    \sum_{m =-(L+J)}^{(L+J)} \abs{  f(U^{-m} x) }^p  w(U^{-m} x) d\mu(x) \\
       & \leq \frac{1}{\lambda^p} \frac{1}{2L+1} \sum_{m =-(L+J)}^{(L+J)} \int_X  \abs{  f(U^{-m} x) }^p  w(U^{-m} x) d\mu(x) \\
   & \leq \frac{1}{\lambda^p} \frac{1}{2L+1} \sum_{m =-(L+J)}^{(L+J)} \int_X  \abs{  f(U^{-m} x) }^p  w(U^{-m} x) d\mu(x) \\
       & = \frac{1}{\lambda^p} \frac{1}{2L+1} \sum_{m =-(L+J)}^{(L+J)} \int_X  \abs{  f( x) }^p  w( x) d\mu(x) \\
    & \leq  \frac{C}{\lambda^p} \frac{1}{2L+1} 2(L+J)+1 \normb{f}^p_{L^p_w(X)} \\
     & \leq  \frac{C}{\lambda^p} (\frac{2L}{2L+1} +  \frac{2J+1}{2L+1} ) \normb{f}^p_{L^p_w(X)} \\
    & \leq \frac{C}{\lambda^p} \normb{f}^p_{L^p_w(X)}
\end{align*}

by choosing $L$ appropriately.  Conclusion $(1)$ of the theorem  now follows by using the Marcinkiewicz interpolation theorem.
 \end{proof}

%\subsection{Ergodic converse}
Now, we prove the converse of Theorem[$\ref{Sak_thm_Ergodic_oneway}$] for $p>1$ with the additional assumptions (1) $(X, \mathcal{B}, \mu)$ is a probability space and (2) $U$ is ergodic measure preserving transformation.
Using transference method, we prove the converse of Theorem[$\ref{Sak_thm_Ergodic_oneway}$]. A direct proof can be seen in \text{\cite{Att1}}. For this we require the concept of ergodic rectangles which we define below \text{\cite{Att1}}.
\begin{definition}[Ergodic Rectangle]
Let $E$ be a subset of $X$ with positive measure and let $K \geq 1$ be such that $ U^i E \cap U^j E = \phi \quad \text{if} \quad  i\neq j $ and $-K \leq i,j \leq K$. Then the set $R = \cup_{i=-K}^K U^i E$ is called ergodic rectangle of length $2K+1$ with base $E$.
\end{definition} 
For the proof of following lemma[$\ref{sak_ergodic_rect}$], refer{\text{\cite{Att1}}}.
\begin{lemma}{\label{sak_ergodic_rect}}
Let $(X,\textbf{B}, \mu)$ be a probability space, $U$ an ergodic invertible measure preserving transformation on $X$ and $K$ a positive integer.
\begin{enumerate}
    \item If $ F \subseteq X $ is a set of positive measure then there exists a subset $E \subseteq F$ of positive measure such that $E$ is base of an ergodic rectangle of length $2K+1$. 
    \item There exists a countable family $\left \{ E_j \right \} $ of bases of ergodic rectangles of length $2K+1$ such that $X = \cup_j E_j$.
\end{enumerate}
\end{lemma}

\begin{theorem}
    Let $(X , \mathcal{B}, \mu)$ be a probability space, U an invertible ergodic measure preserving transformation on $X$. If $\tilde{M} f$ is bounded on $L^p_w(X)$ for some $1 < p < \infty$, then $w \in A_P(X)$.
\end{theorem}
\begin{proof}
For the given function $w$ on $X$, for a.e $x \in X$ define the sequence $w_x(k) = w(U^{-k} x) $.
We shall prove that
\[ esssup_{ x \in X} \biggl( \frac{1}{\abs{I}} \sum_{k \in I} \abs{w_x(k)} \biggr) \biggl( \frac{1}{\abs{I}} \sum_{k \in I} \abs{w_x(k)}^{p^\prime -1} \biggr)^{p-1} \leq C \]
This will prove that $ w \in A_p(X)$. In order to prove this, we shall prove that the Hardy-Littlewood maximal operator $M$ is bounded on $\ell^p_{w_x}(\ZZ)$ and
\[ \normb{Ma}_{\ell^p_{w_x}(\ZZ)} \leq C_p \normb{a}_{\ell^p_{w_x}(\ZZ)} \]  where $C_p$ is independent of $x$.
In order to prove the above inequality, take a sequence $a=\seq{a} \in \ell^p_{w_x}(\ZZ)$. 

Let $R = \cup_{k =-2J}^{2J} U^k E$ be an ergodic rectangle of length $4J+1$ with base $E$. Let $F$ be any measurable subset of $E$. Then $F$ is also base of an
ergodic rectangle of length $4J+1$. Let $R^\prime = \cup_{k=-2J}^{2J} U^k F$.
Define function $f$ and $w$ as follows.
% supported in $R^\prime$, $k =-J \dots J$ 
\begin{align*}
f(U^{-k} x)   = \begin{cases}
	 \,  a(k)  \quad if  \quad x \in F \quad \text{and} -J \leq k \leq  J  \\
	  \, 0 \quad {otherwise} 
	\end{cases}
\end{align*}
%$Further define sequence $w_x(k) = w(U^{-k} x) $ for a.e $ x \in X$.
%and 
%\begin{align*}
%w(U^{-k} x)   = \begin{cases}
%	 \,  w_x(k)  \quad if  \quad x \in F \quad \text{and} -J \leq k \leq  J  \\
%	  \, 0 \quad {otherwise} 
%	\end{cases}
%\end{align*}
%using Lemma
%\[ \tilde{M}_J f(U^{-j} x)  = \sup_{  1 \leq N \leq J} \frac{1}{2N+1} \sum_{k=-N}^N \abs{  f(U^{-j-k}) } = \sup_{  1 \leq N \leq J} \frac{1}{2N+1} \sum_{k=-N}^N a(j+k) = M_J a(j) \]

Then,
   \begin{align*}
       \normb{f}^p_{L^p_w(X)} & = \int_{X} \abs{f(x)}^p w(x) d\mu(x) = \int_{R^\prime} \abs{f(x)}^p w(x) d\mu(x) \\
       & = \sum_{k=-J}^J \int_{U^{k} F} \abs{f(x)}^p w(x) d\mu(x) \\
 & = \sum_{k=-J}^J \int_{ F} \abs{f(U^{-k} x) }^p w(U^{-k} x) d\mu(x) \\
  & = \sum_{k=-J}^J \int_{ F} \abs{ a(k) }^p w_x(k) d\mu(x) \\
& =\int_F \biggl( \sum_{k =-J}^J \abs{a(k)}^p w_x(k) \biggr) d\mu(x) \\
% & = \biggl( \sum_{k=-J}^J \abs{ a(k)}^p w_x(k) \biggr) \mu(F) \\
& \leq \normb{a}_{\ell^p_{w_x}(\ZZ)} \mu(F)
   \end{align*}

%Note for , using Equation[$\ref{Sak_eq2}$] and Lemma[$\ref{Sak_lem_2.11}$], we have
Using  Lemma[$\ref{Sak_lem_2.11}$], it is easy to observe that for $-J \leq m \leq J$ and $ x \in F$
\[ \tilde{M_J} f(U^{-m} x) =M^\prime_J a(m) \geq \frac{1}{3} M_J a(m) \]
%\[ \tilde{M_j}^\star f(U^m x) = \sup_{ 1 \leq N \leq J} \frac{1}{2N+1} \sum_{k =-N}^N \abs{ f(U^{m-k} x)} = \sup_{1 \leq N \leq J} \frac{1}{2N+1} \sum_{k=-N}^n \abs{a(m-k)} = M_j a(m) \] 
Now,
\begin{align*}
C \normb{f}^p_{L^p_w(X)}  & \geq \int_X \abs{\tilde{M_J} f(x) }^p w(x) d\mu(x) \\
& =  \int_{R^\prime} \abs{\tilde{M_J} f(x) }^p w(x) d\mu(x) \\
& =  \sum_{k=-J}^{J} \int_{U^{k}F}  \abs{\tilde{M_J} f(x) }^p w(x) d\mu(x) \\
& =  \sum_{k=-J}^{J} \int_{F}  \abs{\tilde{M_J} f(U^{-k} x) }^p w(U^{-k} x) d\mu(x) \\
& =   \sum_{k=-J}^{J}  \int_{F} \abs{{M^\prime_J a}(k) }^p w_x(k)  d\mu(x) \\
& =  \int_{F}  \sum_{k=-J}^{J}  \abs{{M^\prime_J a}(k) }^p w_x(k)  d\mu(x) \\
%& =   \int_{F} \sum_{k=-2J}^{2J}  \abs{\tilde{M_J a}(k) }^p w_x(k)  d\mu(x) \\
& \geq \frac{1}{3} \int_{F} \sum_{k=-J}^{J}  \abs{{M_J a}(k) }^p w_x(k)  d\mu(x) 
   \end{align*}
So from the above estimates
\[ \frac{1}{\mu(F)}  \int_{F} \sum_{k=-J}^{J}  \abs{{M_J a}(k) }^p w_x(k)  d\mu(x) \leq C \normb{a}_{\ell^p_{w_x}(\ZZ)} \]
Since $F$ was an arbitrary subset of $E$, we get 
\[    \sum_{k=-J}^{J}  \abs{{M_J a}(k) }^p w_x(k)   \leq C \normb{a}_{\ell^p_{w_x}(\ZZ)} \] a.e $x \in E$. Since $U$ is ergodic, $X$ can be written as countable union of bases of ergodic rectangles of length $4J+1$. Therefore for a.e $ x \in X$,
\[    \sum_{k=-J}^{J}  \abs{{M_J a}(k) }^p w_x(k)   \leq C \normb{a}_{\ell^p_{w_x}(\ZZ)} \]
%\[ \sum_{k=-J}^J \biggl| \sup_{J \geq N \geq 1 } \frac{1}{2N+1} \sum_{k=-N}^N \abs{a(k) } \biggr|^p w_x(k) \leq C \normb{a}_{\ell^p_{w_x}(\ZZ)}\]
Since $C$ is independent of $J$, a.e $x \in X$,
%\[ \sum_{k \in \ZZ} \biggl| \sup_{ N \geq 1 } \frac{1}{2N+1} \sum_{k=-N}^N \abs{a(k) }\biggr|^p w_x(k) \leq C \normb{a}_{\ell^p_{w_x}(\ZZ)}\]
\[    \sum_{k \in \ZZ}  \abs{{M a}(k) }^p w_x(k)   \leq C \normb{a}_{\ell^p_{w_x}(\ZZ)} \]
It follows that the sequence $\seq{w_x}$ as defined by $w_x(k) = w(U^k x)$ belongs to $A_p(\ZZ)$ a.e $x \in X$ and $A_p$ weight constant for $w_x$ is independent of $x$ so that $w \in A_p(X)$.
\end{proof}

\section{Conclusion} 
The study of  maximal ergodic operator on $L^p_w( X, \mathcal{B}, \mu)$ spaces paves the way to study this operator on variable $L^{\pdot}(X,\mathcal{B},\mu)$ spaces. Using Rubio de Francia extrapolation method ${\text{\cite{David5}}}$ and appropriate variable Holder's inequality, we hope to achieve this result.
%Theorem[$\ref{Sak_var_ergodic_theorem}$] is proved with the assumption that $(X,\mathcal{B}, \mu)$ being a probability space. Lemma[$\ref{Sak_fior_var_book}$] requires that $\mu(X)$ is finite.  We hope that one can prove Theorem[$\ref{Sak_var_ergodic_theorem}$]  with the assumption that  $(X,\mathcal{B}, \mu)$  is a $\sigma$-finite measure space, by suitably splitting $h(X)$ and with appropriate use of variable Holder's inequality. This work is for future investigation.
%\bibliographystyle{abbrv}
%\bibliography{NSS14}

\end{document}